\documentclass[12pt]{article}

\usepackage{amssymb, enumerate}
\usepackage{amsmath}
\usepackage{amsthm}

\newcounter{itemcounter}
\numberwithin{itemcounter}{section}

\newtheorem{thm}[itemcounter]{Theorem}

\newtheorem{defi}[itemcounter]{Definition}
\newtheorem{prop}[itemcounter]{Proposition}
\newtheorem{cor}[itemcounter]{Corollary}

\newtheorem{con}[itemcounter]{Conjecture}

\newtheorem*{thm*}{Theorem}
\newtheorem*{con*}{Conjecture}
\newtheorem*{cor*}{Corollary}
\newtheorem*{ack*}{Acknowledgements}

\newcommand{\Irr}{\mathop{\rm Irr}\nolimits}
\newcommand{\IBr}{\mathop{\rm IBr}\nolimits}

\newcommand{\Aut}{\mathop{\rm Aut}\nolimits}
\newcommand{\Out}{\mathop{\rm Out}\nolimits}

\newcommand{\OF}{\mathop{\rm f_\mathcal{O}}\nolimits}
\newcommand{\SOF}{\mathop{\rm sf_\mathcal{O}}\nolimits}
\newcommand{\mf}{\mathop{\rm mf}\nolimits}

\newcommand{\nth}{\mathop{\rm th}\nolimits}

\newcommand{\lcm}{\mathop{\rm lcm}\nolimits}

\newcommand{\cO} {\mathcal{O}}

\newcommand{\NN} {\mathbb{N}}

\newenvironment{enumerate*}{%
     \begin{enumerate}%
     }%
    {\end{enumerate}}

\begin{document}

\title{Donovan's conjecture and blocks with abelian defect groups \footnote{This research was supported by the EPSRC (grant no. EP/M015548/1).} }

\author{Charles W. Eaton\footnote{School of Mathematics, University of Manchester, Manchester, M13 9PL, United Kingdom. Email: charles.eaton@manchester.ac.uk} and Michael Livesey\footnote{School of Mathematics, University of Manchester, Manchester, M13 9PL, United Kingdom. Email: michael.livesey@manchester.ac.uk}}

\date{24th February, 2018}
\maketitle


\begin{abstract}
We give a reduction of Donovan's conjecture for abelian groups to a similar statement for quasisimple groups. Consequently we show that Donovan's conjecture holds for abelian $2$-groups.
\end{abstract}


\section{Introduction}

Let $(K,\mathcal{O},k)$ be a $p$-modular system with $k$ algebraically closed. We are interested in the following conjecture in the case of abelian $p$-groups:

\begin{con}[Donovan]
\label{Donovan:conj}
Let $P$ be a finite $p$-group. Then amongst all finite groups $G$ and blocks $B$ of $kG$ with defect groups isomorphic to $P$ there are only finitely many Morita equivalence classes.
\end{con}

One approach to the conjecture is reduction to quasisimple groups followed by the application of the classification of finite simple groups. For example, in~\cite{ekks14} it was proved that Donovan's conjecture holds for elementary abelian $2$-groups in this way, following a partial reduction in~\cite{du04}. The reason that this result could not be extended to arbitrary abelian $2$-groups is that it was not known how Morita equivalence classes of blocks relate to those for blocks of normal subgroups of index $p$ in general, with only the special case of a split extension being known by~\cite{kk96}. Our approach uses~\cite{ke05}, where it was shown that for each $p$-group $P$, Donovan's conjecture is equivalent to both of the following conjectures holding, the first originating from a question of Brauer:

\begin{con}[Weak Donovan]
\label{weak_Donovan}
Let $P$ be a finite $p$-group. Then there is $c(P) \in \NN$ such that if $G$ is a finite group and $B$ is a block of $kG$ with defect groups isomorphic to $P$, then the entries of the Cartan matrix of $B$ are at most $c(P)$.
\end{con}

For a finite group $G$ and $n\in\mathbb{Z}$, let $-^{(p^n)}$ be the ring automorphism of $kG$ defined by $(\sum_{g \in G} a_g g)^{(p^n)} = \sum_{g \in G} (a_g)^{p^n} g$. Note that $-^{(p^n)}$ permutes the blocks of $kG$. Define the Morita-Frobenius number $\mf(B)$ of a block $B$ of $kG$ to be the smallest $n \in \NN$ such that $B^{(p^n)}$ is Morita equivalent to $B$. An equivalent definition is given in Definition \ref{morita_frobenius:def}.

\begin{con}[Kessar]
\label{Kessar:conj}
Let $P$ be a finite $p$-group. Then there is $m(P) \in \NN$ such that if $G$ is a finite group and $B$ is a block of $kG$ with defect groups isomorphic to $P$, then $\mf(B) \leq m(P)$.
\end{con}

Our approach is to use the related notion of the \emph{strong $\cO$-Frobenius} number $\SOF(B)$ of a block $B$ as defined in~\cite{el18} and show that the question of uniformly bounding the Morita-Frobenius number may be reduced to bounding the strong $\cO$-Frobenius number for quasisimple groups. Conjecture \ref{weak_Donovan} for abelian groups is already reduced to blocks of quasisimple groups in~\cite{du04}.


Combining these two reductions we obtain our main result, which is that Donovan's conjecture for abelian defect groups reduces to bounding the strong $\cO$-Frobenius number and Cartan invariants for blocks of quasisimple groups:

\begin{thm}
\label{reduce:theorem}
If there are functions $s,c:\mathbb{N}\to\mathbb{N}$ such that for all $\cO$-blocks $B$ of quasisimple groups with abelian defect groups of order $p^d$, $\SOF(B) \leq s(d)$ and all Cartan invariants are at most $c(d)$, then Donovan's conjecture holds for $k$-blocks with abelian defect groups.
\end{thm}

Donovan's conjecture may also be stated for $\cO$-blocks. Theorem \ref{reduce:theorem} may be phrased in terms of this stronger conjecture as follows, but we are not able to reduce it to quasisimple groups:

\begin{cor}
If Donovan's conjecture holds for $\cO$-blocks of quasisimple groups with abelian defect groups, then it holds for all $k$-blocks with abelian defect groups.
\end{cor}

For $p=2$ we then apply~\cite{ekks14} to show:

\begin{thm}
\label{maintheorem}
Donovan's conjecture holds for abelian $2$-groups.
\end{thm}

To the authors' knowledge, the Weak Donovan conjecture and bounding the strong $\cO$-Frobenius numbers of blocks of quasisimple groups are not sufficient to prove that the $\cO$ version of Donovan's conjecture holds for abelian $2$-groups.

The structure of the paper is as follows. In Section 2 we recall some preliminaries about strong $\cO$-Frobenius numbers from~\cite{el18} and Donovan's conjecture. In Section 3 we gather together the reductions of D\"{u}vel~\cite{du04} and the authors~\cite{el18} to ultimately obtain a reduction for Donovan's conjecture, in the case of abelian defect groups, to blocks of quasisimple groups. In Section 4 we use this reduction to prove Donovan's conjecture for abelian defect groups in characteristic $2$.


\section{Strong $\cO$-Frobenius numbers and Donovan's conjecture}

In this section we recall the definition of the strong $\cO$-Frobenius number, collect together some of its properties and describe how it relates to Donovan's conjecture. Let $G$ be a finite group and $B$ a block of $\cO G$. We denote by $\Irr(G)$ the set of irreducible characters of $G$ and $\Irr(B)$ the subset of $\Irr(G)$ of irreducible characters lying in the block $B$. We write $kB$ for the block of $kG$ corresponding to $B$ and $KB$ for the $K$-subspace of $KG$ generated by $B$. We denote by $e_B\in\cO G$ the block idempotent corresponding to $B$ and $e_{kB}\in kG$ the block idempotent corresponding to $kB$. Finally for each $\chi\in\Irr(G)$ we denote by $e_\chi\in KG$ the character idempotent corresponding to $\chi$. Note that $KB=\bigoplus_{\chi\in\Irr(B)}KGe_\chi$. If $A$ and $B$ are finitely generated $R$-algebras for $R \in \{ K,\cO ,k \}$, we write $\mod(A)$ for the category of finitely generated $A$-modules and $A\sim_{\operatorname{Mor}} B$ if $\mod(A)$ and $\mod(B)$ are (Morita) equivalent as $R$-linear categories.
\newline
\newline
We quote the following definition from~\cite[Definition 3.2]{el18}.

\begin{defi}
\label{twist:def}
Let $q$ be a, possibly zero or negative, power of $p$. We denote by $-^{(q)}:k\to k$ the field automorphism given by $\lambda\mapsto\lambda^{\frac{1}{q}}$. Let $A$ be a $k$-algebra. We define $A^{(q)}$ to be the $k$-algebra with the same underlying ring structure as $A$ but with a new action of the scalars given by $\lambda.a=\lambda^{(q)}a$, for all $\lambda\in k$ and $a\in A$. For $a\in A$ we define $a^{(q)}$ to be the element of $A$ associated to $a$ through the ring isomorphism between $A$ and $A^{(q)}$. For $M$ an $A$-module we define $M^{(q)}$ to be the $A^{(q)}$-module associated to $M$ through the ring isomorphism between $A$ and $A^{(q)}$.

Note that for $G$ a finite group, we have $kG\cong kG^{(q)}$ as we can identify $-^{(q)}:kG\to kG$ with the ring isomorphism:
\begin{align*}
-^{(q)}:kG&\to kG\\
\sum_{g\in G}\alpha_gg&\mapsto\sum_{g\in G}(\alpha_g)^qg.
\end{align*}
If $B$ is a block of $kG$ then we can and do identify $B^{(q)}$ with the image of $B$ under the above isomorphism.

By an abuse of notation we also use $-^{(q)}$ to denote the field automorphism of the universal cyclotomic extension of $\mathbb{Q}$ defined by $\omega_p\omega_{p'}\mapsto\omega_p\omega_{p'}^{\frac{1}{q}}$, for all $p^{\nth}$-power roots of unity $\omega_p$ and $p'^{\nth}$ roots of unity $\omega_{p'}$. If $\chi\in\Irr(G)$, then we define $\chi^{(q)}\in\Irr(G)$ to be given by $\chi^{(q)}(g)=\chi(g)^{(q^{-1})}$ for all $g\in G$ and we define $\varphi^{(q)}$ for $\varphi\in\IBr(G)$ in an analogous way. Note that if $S$ is a simple $kG$ module with Brauer character $\varphi$, then $S^{(q)}$ is a simple $kG$ module with Brauer character $\varphi^{(q)}$. If $B$ is a block of $\cO G$ with $\chi\in\Irr(B)$, then we define $B^{(q)}$ to be the block of $\cO G$ with $\chi^{(q)}\in\Irr(B^{(q)})$. By considering a simple $B$-module and the decomposition matrix of $\cO G$ we see that $(kB)^{(q)}=k(B^{(q)})$, in particular $B^{(q)}$ is well-defined.
\end{defi}

\begin{defi}
\label{morita_frobenius:def}
The \textbf{Morita Frobenius number} $\mf(A)$ of a finite dimensional $k$-algebra $A$ is the smallest integer $n$ such that $A\sim_{\operatorname{Mor}}A^{(p^n)}$ as $k$-algebras. Let $G$ be a finite group and $B$ a block of $\cO G$. The \textbf{$\cO$-Frobenius number} $\OF(B)$ of $B$ is the smallest integer $n$ such that $B\cong B^{(p^n)}$ as $\cO$-algebras and the \textbf{strong $\cO$-Frobenius number} $\SOF(B)$ of $B$ is the smallest integer $n$ such that there exists an $\cO$-algebra isomorphism $B\to B^{(p^n)}$ such that the induced bijection of characters is given by $\chi\mapsto\chi^{(p^n)}$ for all $\chi\in\Irr(B)$.
\end{defi}

We gather together some results concerning the invariants defined above:

\begin{prop}
\label{morita_O_proposition}
Let $G$, $H$ and $N$ be finite groups with $N \lhd G$, and let $B$, $C$ and $b$ be blocks of $\cO G$, $\cO H$ and $\cO N$ respectively with $B$ covering $b$. Let $D$ be a defect group for $B$.

\begin{enumerate}[(i)]
\item $\mf(kB)\leq\OF(B)\leq\SOF(B)\leq|D|^2!\OF(B)$.
\item If $B$ and $C$ are Morita equivalent, then $\SOF(B)=\SOF(C)$.
\item If $B$ has abelian defect groups and $[G:N]=p$, then $\SOF(B)\leq\SOF(b)$.
\item If $Z \leq G$ is a central $p$-subgroup and $B_Z$ is the unique block of $\cO (G/Z)$ corresponding to $B$ via the natural homomorphism $\cO G \rightarrow \cO (G/Z)$, then $\SOF(B_Z) \leq \SOF(B)$.
\item $\SOF(B \otimes_{\cO} C) \leq \lcm\{\SOF(B),\SOF(C)\}$, viewing $B \otimes_{\cO} C$ as a block of $\cO(G\times H)$.
\item If $D \lhd G$, then $\SOF(B) \leq n^{\frac{1}{2} (\log_2(n)-1)}$, where $n=|\Aut(D)|_{p'}$.
\item If $B$ is Morita equivalent to a principal block, then $\SOF(B) \leq (|D|^2)!$.

\end{enumerate}
\end{prop}

\begin{proof}
\begin{enumerate}[(i)]
\item The first two inequalities should be clear from the definitions and the final inequality is in~\cite[Proposition 3.11]{el18}.
\item This is~\cite[Proposition 3.12]{el18}.
\item This is~\cite[Theorem 3.16]{el18}.
\item This is~\cite[Proposition 3.17]{el18}.
\item This is clear from the definition.
\item This is~\cite[Lemma 5.1]{el18}.
\item This follows from (i) and (ii).
\end{enumerate}
\end{proof}


The reason we work with the Morita Frobenius and strong $\cO$-Frobenius numbers is the following, which allows us to divide Donovan's conjecture into two parts. We state Kessar's result in a slightly stronger form than that given in~\cite{ke05}, but the proof is the same.

\begin{thm}\cite[Theorem 1.4]{ke05}
\label{kessar}
Consider a class $\mathcal{X}$ of blocks of finite groups with defect group isomorphic to $P$. To prove Donovan's conjecture for all $k$-blocks in $\mathcal{X}$ it is enough to bound $\mf(B)$ and all the Cartan invariants of all blocks in $\mathcal{X}$.
\end{thm}

As mentioned in the introduction, we work with the strong $\cO$-Frobenius number rather than the Morita Frobenius number in order to compare blocks with those of normal subgroups of index $p$ (as in Proposition \ref{morita_O_proposition}). Since Theorem \ref{kessar} is only known to hold for $k$-blocks at present, despite initially working with $\cO$-blocks we are only able to draw conclusions about Donovan's conjecture for $k$-blocks.



\section{Reductions for Donovan's conjecture}

We begin with a result derived from the proof of the main result of~\cite{kk96} which states that abelian defect groups act as inner automorphisms on stable blocks of normal subgroups.

\begin{prop}
\label{abelian_defect_action:prop}
Let $G$ be a finite group and $B$ a block of $kG$ with abelian defect group $D$, $N\lhd G$ and $b$ a $G$-stable block of $kN$ covered by $B$. Then $D\leq G[b]$.
\end{prop}

\begin{proof}
We set $H:=DN\leq G$. Certainly $b$ is $H$-stable and so, as $[H:N]$ is a power of $p$, there is a unique block $B_H$ of $kH$ covering $b$ and $e_{B_H}=e_b$. Since $b$ is $G$-stable it is the unique block of $kN$ covered by $B$ and so $e_b$ acts as the identity on any $B$-module. Therefore, as $e_{B_H}=e_b$, the restriction of any $B$-module to $H$ lies entirely in $B_H$. Now let $M$ be an idecomposable $B$-module with vertex $D$, the existence of which is guaranteed by~\cite[$\S14$ Corollary 5]{alp}. Then, since $D\leq H$, there must exist a summand of $M\downarrow_H$ with vertex $D$. Therefore $B_H$ must have a defect group $Q$ containing $D$ and, since $B_H$ covers $b$ and $b$ has defect group $D\cap N$, we must have that $|Q\cap N|=|D\cap N|$. Therefore
\begin{align*}
|Q|=[Q:Q\cap N].|Q\cap N|=[QN:N].|D\cap N|=[DN:N].|D\cap N|=|D|
\end{align*}
and so $Q=D$. We have just proved that $B_H$ has abelian defect group and so, by the first half of the proof of the main theorem in~\cite{kk96}, $H=H[b]$ and $D\leq G[b]$.
\end{proof}

In~\cite{du04} Donovan's conjecture for abelian $p$-groups is reduced to groups with a normal central product of normal quasisimple groups and solvable quotient. A major problem in reducing further is to eliminate the case that there is a normal subgroup of index $p$. To reduce all the way to quasisimple groups we consider Conjectures \ref{Donovan:conj} and \ref{Kessar:conj} separately, and show that the question of bounds on the Morita-Frobenius number may be reduced to quasisimple groups. Since bounding strong $\cO$-Frobenius numbers is stronger than bounding Morita Frobenius numbers, we are not showing that Donovan's conjecture may be reduced to quasisimple groups for abelian $p$-groups.

First we give the reduction to blocks of $kG$ for groups $G$ with certain constraints. This is extended from~\cite{du04}. Recall that a block is quasiprimitive if for every normal subgroup, it covers a unique (i.e., stable) block. In the following, recall that the Fitting subgroup $F(G)$ is the product of the $O_r(G)$ for all primes $r$. The layer $E(G)$ of $G$ is the product of the components of $G$ and the generalized Fitting subgroup $F^*(G)$ is $E(G)F(G)$.

\begin{prop}
\label{duvel}
Let $P$ be an abelian $p$-group for a prime $p$. In order to verify Donovan's conjecture for $P$, it suffices to verify that there are only a finite number of Morita equivalence classes of blocks $B$ of $kG$ with defect group $D \cong P$ for finite groups $G$ satisfying the following conditions:
\begin{enumerate}[(i)]
\item $F(G)=Z(G)=O_p(G)O_{p'}(G)$;
\item $O_{p'}(G) \leq [G,G]$;
\item $G= \langle D^g:g \in G \rangle$;
\item every component of $G$ is normal in $G$;
\item if $H \leq G$ is a component, then $H \cap D \neq Z(H) \cap D$;
\item if $H$ is any characteristic subgroup of $G$, then $B$ covers a unique (i.e., $G$-stable) block $b$ of $kH$ and further $G[b]=G$.
\end{enumerate}
\end{prop}

\begin{proof}
Parts (i)-(v) are~\cite[Theorem 1.11]{du04}, and for part (vi) we examine its proof a little further. Now~\cite[Theorem 1.11]{du04} is a consequence of~\cite[Proposition 1.12]{du04}, which says that if $G_0$ is a finite group and $B_0$ is a block of $kG_0$ with defect group $D_0 \cong P$, then there is a finite group $G_1$ and a quasiprimitive block $B_1$ of $kG_1$ with defect group $D_1 \cong P$ and $B_1$ Morita equivalent to $B_0$ such that the unique block $B$ of $kG$, where $G:= \langle {}^gD_1:g \in G_1 \rangle$, satisfies conditions (i)-(v) (noting that $D_1$ is also a defect group for $B$). Since $B_1$ is quasiprimitive, the first part of (vi) holds for $B$. Let $H$ be a characteristic subgroup of $G$ and let $b$ be the unique block of $kH$ covered by $B_1$ (and $B$). Now by Proposition \ref{abelian_defect_action:prop} $D \leq G[b]$, and since $G[b]\lhd G$, (iii) implies that $G=G[b]$.
\end{proof}

\begin{defi}
We call a block $B$ of $\cO G$, where $G$ is a finite group, \emph{reduced} if it satisfies the conditions of Proposition \ref{duvel}.
\end{defi}

\bigskip

{\noindent \emph{Proof of Theorem \ref{reduce:theorem}.}}

Suppose that there is a function $s:\mathbb{N}\to\mathbb{N}$ such that $\SOF(b) \leq s(d)$ for all $\cO$-blocks $b$ of quasisimple groups with abelian defect group of order $p^d$.

Let $B$ be a block of $\cO G$ with abelian defect group $D$ such that $kB$ satisfies conditions (i)-(vi) of Proposition \ref{duvel}. We examine the structure of $G$. Write $L_1,\ldots,L_t$ for the components, so the layer $E(G)=L_1 \cdots L_t$. The generalized Fitting subgroup $F^*(G)$ is $E(G)F(G)$, and by~\cite[31.13]{asc00} $C_G(F^*(G)) \leq F^*(G)$, so $C_G(F^*(G))=Z(F^*(G))=Z(G)$ since $F(G)=Z(G)$. Writing $\overline{G}=G/Z(G)$, we have $\overline{G} \leq \Aut(\overline{L_1} \times \cdots \times \overline{L_1})$. Since $L_i \lhd G$ for each $i$, we further have $\overline{G} \leq \Aut(\overline{L_1}) \times \cdots \times \Aut(\overline{L_t})$. Since by Schreier's conjecture each $\Out(L_i)$ is solvable, we have $G/E(G)$ solvable.

Consider the series $$1 \lhd O_p(G/E(G)) \lhd O_{p,p'}(G/E(G)) \lhd O_{p,p',p}(G/E(G)) \lhd \cdots \lhd G/E(G)$$ as defined in~\cite[6.3]{gor}, and write $$E(G) = G_0 \lhd G_1 \lhd \cdots \lhd G_n=G$$ for the series of preimages in $G$. Note that $G_i$ is a characteristic subgroup of $G$ for each $i$. Write $B_i$ for the unique block of $\cO G_i$ covered by $B$. If $B_{i+1}/B_i$ is a $p$-group, then by Proposition \ref{morita_O_proposition} $\SOF(B_{i+1}) \leq \SOF(B_i)$. If $B_{i+1}/B_i$ is a $p'$-group, then since $G_{i+1}[B_i]=G_{i+1}$ by~\cite[Proposition 2.2]{kkl12} $B_i$ and $B_{i+1}$ are Morita equivalent, and so by Proposition \ref{morita_O_proposition} $\SOF(B_{i+1}) = \SOF(B_i)$. We have shown that $\SOF(B) \leq \SOF(B_0)$. Now $E(G)$ is a central product of $L_1,\ldots,L_t$. Write $b_i$ for the unique block of $\cO L_i$ covered by $B_0$. There is a subgroup $Z$ of the centre of $H:=L_1 \times \cdots \times L_t$ such that $E(G) \cong H/Z$. Let $B_H$ be the unique block of $\cO H$ with $Z_{p'}:=O_{p'}(Z)$ in its kernel corresponding to $B_0$ via the natural homomorphisms $\cO G \rightarrow \cO (G/Z_{p'}) \rightarrow \cO (G/Z)$. Then by Proposition \ref{morita_O_proposition} $\SOF(B_0) \leq \SOF(B_H)$. Now $B_H = b_1 \otimes \cdots \otimes b_t$, and so by Proposition \ref{morita_O_proposition} $$\SOF(B) \leq \SOF(B_H) \leq \lcm\{\SOF(b_1),\ldots,\SOF(b_t)\} \leq \prod_{i=1}^t \SOF(b_i) \leq \prod_{i=1}^t s(\log_p(|D_i|)),$$ where $D_i = L_i \cap D$ is a defect group for $b_i$. We have shown that $\SOF(B)$ is bounded in terms of $D$. Hence by Proposition \ref{morita_O_proposition} $\mf(B)$ is bounded in terms of $D$ for all reduced blocks $B$.

We have assumed that the Cartan invariants of the blocks of quasisimple groups with abelian defect groups are bounded in terms of the defect. Then by~\cite[Theorem 3.2]{du04} (together with the observation that the Cartan invariants are bounded in terms of the dimensions of the Ext spaces for pairs of simple modules and the Loewy length) the Cartan invariants of any block with abelian defect groups are bounded in terms of the defect, and so in particular this holds for the class of reduced blocks. Hence by Theorem \ref{kessar} Donovan's conjecture holds for reduced blocks and by Proposition \ref{duvel} we are done.





\begin{cor}
For a fixed prime $p$, if Donovan's conjecture holds for $\cO$-blocks of quasisimple groups with abelian defect groups, then it holds for all $k$-blocks with abelian defect groups.
\end{cor}


\section{Donovan's conjecture for abelian $2$-groups}

\begin{thm}[\cite{ekks14}]
\label{component_classification}
Let $G$ be a quasi-simple group.  If $B$ is a $2$-block of $\mathcal{O}G$ with  abelian defect group $D$,  then one (or more) of the following holds:

(i) $G/Z(G)$ is one of  $A_1(2^a)$,   $\,^2G_2(q)$ (where $q \geq 27$ is a power of $3$ with odd exponent), or $J_1$,  $B$  is the principal block and $D$ is elementary abelian;

(ii) $G$  is $Co_3 $,   $B$ is a non-principal block, $D \cong C_2 \times C_2 \times  C_2 $ (there is one such block);

(iii) $B$ is a nilpotent covered block;

(iv) $G$ is of type $D_n(q)$ or $E_7(q)$, where $n=2t$ for $t$ odd and $q$ is a power of an odd prime. $B$ is Morita equivalent  to a block  $C$ of a   subgroup   $L=L_0 \times L_1$   of $G$   as  follows:   The  defect groups   of $C$ are isomorphic to  $D$,    $L_0$   is abelian  and the  $2$-block of  $L_1 $ covered by $C$ has  Klein $4$-defect groups.

(v) $B$ has Klein four defect groups.
\end{thm}

\begin{proof}
This is taken from Propositions 5.3 and 5.4 and Theorem 6.1 of~\cite{ekks14}.
\end{proof}

\begin{thm}
Donovan's conjecture holds for abelian $2$-groups.
\end{thm}

\begin{proof}
By Theorem \ref{reduce:theorem} it suffices to show that there is bound for $\SOF(B)$ for $\cO$-blocks $B$ with abelian defect groups of quasisimple groups in terms of the defect group $D$. We check each of the cases in Theorem \ref{component_classification} in turn. In cases (i), (iv) and (v) $B$ is morita equivalent to a principal block, so $\SOF(B) \leq (|D|^2)!$ by Proposition \ref{morita_O_proposition}. Similarly in case (ii), $B=B^{(2)}$ so again we have $\SOF(B) \leq (|D|^2)!$. Finally, if $B$ is nilpotent covered, then by~\cite{pu11} $B$ is Morita equivalent to a block with a normal defect group. Then by Proposition \ref{morita_O_proposition} $\SOF(B) \leq n^{\frac{1}{2} \log_2(n)-1}$, where $n=|\Aut(D)|_{2'}$. The theorem follows.
\end{proof}


\end{document}